\documentclass[12pt,a4paper]{amsart}
\usepackage{}

\makeatletter
\tagsleft@false
\makeatother

\usepackage{mathrsfs}
\usepackage{tikz}
\usepackage{amsfonts}
\usepackage{amsmath}
\usepackage{amssymb}
\usepackage{enumerate}
\usepackage{hyperref}
\usepackage{latexsym}
\usepackage{xcolor}
\usepackage[shortlabels]{enumitem}
\usepackage[numbers,sort&compress]
{natbib}
\setlist[enumerate,1]{label={\upshape(\roman*)}}

 \makeatletter
    
    \newcommand{\Rmnum}[1]
    {\expandafter\@slowromancap\romannumeral #1@}
    \makeatother

\newtheorem{thm}{Theorem}[section]

\newtheorem{prop}[thm]{Proposition}
\newtheorem{lemma}[thm]{Lemma}

\newtheorem{cor}[thm]{Corollary}

\newtheorem{example}[thm]{Example}
\newtheorem{defin}[thm]{Definition}

\theoremstyle{definition}
\newtheorem{remark}[thm]{Remark}

\title[(Total) Perfect codes in (extended) subgroup sum graphs]{(Total) Perfect codes in (extended) subgroup sum graphs}

\date{}

\thanks{*Corresponding author}

\author[Ma]{Xuanlong Ma}
\address{School of Science\\Xi'an Shiyou University\\Xi'an 710065\\China}
\email{xuanlma@xsyu.edu.cn}

\author[Yang]{Yuefeng Yang*}
\address{School of Science\\China University of Geosciences\\Beijing 100083\\China}
\email{yangyf@cugb.edu.cn}

\author[Zhai]{Liangliang Zhai}
\address{School of Science\\Xi'an Shiyou University\\Xi'an 710065\\China}
\email{zhailiang111@126.com}

\begin{document}

\begin{abstract}
Given a finite group $G$ with identity $e$ and a normal subgroup $H$ of $G$, the subgroup sum graph $\Gamma_{G,H}$ (resp. extended subgroup sum graph $\Gamma_{G,H}^+$) of $G$ with respect to $H$ is the graph with vertex set $G$, in which distinct vertices $x$ and $y$ are adjacent whenever $xy\in H\setminus \{e\}$ (resp. $xy\in H$). A group $G$ is said to be {\em code-perfect} if for any normal subgroup $H$ of $G$, $\Gamma_{G,H}$ admits a perfect code. In this paper, we give a necessary and sufficient condition for which normal subgroups $H$ of $G$ satisfy that a (extended) subgroup sum graph of $G$ with respect to $H$ admits a (total) perfect code, and classify all code-perfect Dedekind groups. As an application, we classify all normal subgroups such that the subgroup sum graph of a cyclic group, a dihedral group or a dicyclic group with respect to such a normal subgroup admits perfect codes, respectively. We also determine all abelian groups $A$ and
subgroups $H$ of $A$ such that $\Gamma_{A,H}$ admits a total perfect code.
\end{abstract}

\keywords{Perfect code; Total perfect code; code-perfect, subgroup sum graph; extended subgroup sum graph}

\subjclass[2010]{05C25}.

\maketitle
\section{Introduction}

In this paper, every group considered is finite, and every graph considered is finite and simple. For a graph $\Gamma$ with vertex set $V$, a {\em code} in $\Gamma$ is simply a subset of $V$. A code $C$ of $\Gamma$ is called a {\em perfect code} \cite{Kr}
of $\Gamma$ if every vertex of $\Gamma$ is at distance no more than one to exactly one vertex in $C$. In other words, $C$ is a {\em perfect code} in $\Gamma$ provided that $C$ is independent in $\Gamma$ and every vertex of $V\setminus C$ is adjacent to precisely one vertex of $C$. A code $C$ of $\Gamma$ is said to be a {\em total
perfect code} \cite{GAA} if every vertex of $\Gamma$ has exactly one neighbour in $C$. In some references, a perfect code is also
called an {\em efficient dominating set} \cite{DYP,DYP2,AB23,YSK22,YSK20,DI17,TT13} or {\em independent perfect dominating set} \cite{WX21,Le}, and a total perfect code is also called an efficient open dominating set \cite{HTW}. In the past a few years, perfect codes and total perfect codes in Cayley graphs have attracted considerable attention, see, for example \cite{FHZ,CWZ,HXZ18,MWWZ,ZZ20,ZZ21,Zh,Z15,Kh23,Kh24,YYF}.

Cayley graphs are excellent models for interconnection networks. Hence, there are
many investigations in connection with parallel processing. The definition of the Cayley graph was introduced by Arthur Cayley in 1878 to explain the concept of abstract groups which are described by a set of generators. A variant of Cayley graphs are the Cayley sum graphs. The Cayley sum graph is first defined for an abelian group in \cite{CFRK} and then it is generalized to any arbitrary group in \cite{AM16}. 
The perfect codes and total perfect codes in Cayley sum graphs have been studied by a number of authors \cite{Wa24,Zh24,Ma20,Ma22}.

In \cite{Cam22}, Cameron, Raveendra Prathap and Tamizh Chelvam extended the process of generalization of Cayley graph, by introducing the subgroup sum graph and extended subgroup sum graph for an abelian group. Also in \cite{Cam22}, they studied perfectness, clique number and independence number, connectedness, diameter, spectrum, and domination number of these graphs and their complements.

One can generalize the concepts of subgroup sum graphs and extended subgroup sum graphs over arbitrary groups. Let $G$ be a group with identity element $e$. Two elements $a,b\in G$ are {\em conjugate} in $G$ if there exists $x\in G$ such that $x^{-1}ax=b$. The set of all conjugates of element $a\in G$ is called the {\em conjugacy class} of $a$, and is denoted by $a^G$. 
A subgroup $H$ of $G$ is {\em normal} if $xH=Hx$ for any $x\in G$. A group is called a {\em Dedekind group} if all its subgroups are normal. 
For a normal subgroup $H$ of $G$, the {\em subgroup sum graph} $\Gamma_{G,H}$ (resp. {\em extended subgroup sum graph} $\Gamma_{G,H}^+$) of $G$ with respect to the normal subgroup $H$ is the graph with vertex set $G$, and distinct vertices $x$ and $y$ are adjacent whenever $xy\in H\setminus \{e\}$ (resp. $xy\in H$). Since $e^G=\{e\}$ and $yx=y(xy)y^{-1}$, the condition that $H$ is a normal subgroup ensures that both $\Gamma_{G,H}$ and $\Gamma_{G,H}^+$ are undirected simple graphs.

In this paper, we focus on subgroup sum graphs and extended subgroup sum graphs having (total) perfect codes. It may happen that every subgroup sum graph of a given group admits a perfect code. We call a group with this property a code-perfect group. More explicitly, a group $G$ is  said to be {\em code-perfect} if for any normal subgroup $H$ of $G$, $\Gamma_{G,H}$ admits a perfect code. It is natural to ask which groups are code-perfect. We will answer this
question for Dedekind groups.
The remainder of this paper is organized as follows.
In Section 2, we give some results concerning subgroup sum graphs and extended subgroup sum graphs. In Section 3, we give a necessary and sufficient condition for which normal subgroups $H$ of a given group $G$ satisfy that $\Gamma_{G,H}$ admits a perfect code, and determine all normal subgroups such that the subgroup sum graph of a cyclic group, a dihedral group or a dicyclic group with respect to such a normal subgroup admits perfect codes. In Section 4, we classify all code-perfect Dedekind groups. In Section 5, we give a necessary and sufficient condition for which normal subgroups $H$ of a given group $G$ satisfy that $\Gamma_{G,H}$ admits a perfect code, and determine all abelian groups $A$ and its subgroups $H$ such that $\Gamma_{A,H}$ admits a total perfect code. In Section 6, we give a necessary and sufficient condition for which normal subgroups $H$ of a given group $G$ satisfy that $\Gamma_{G,H}$ admits a (total) perfect  code.

\section{Preliminaries}

In the reminder of this paper, we always assume that $G$ is a multiplicatively written group with identity $e$. In this section, we give some basic results concerning subgroup sum graphs and extended subgroup sum graphs, which will be used frequently.


\begin{lemma}\label{lem-1}
Let $H$ be a normal subgroup of $G$, and $x$ an element of $G$. Then the following are equivalent:
\begin{enumerate}
\item\label{lem-1-1} $x^2\in H$;

\item\label{lem-1-2} $Hx=Hx^{-1}$;

\item\label{lem-1-4} $Hx$ is inverse-closed;

\item\label{lem-1-3} $y^2\in H$ for all $y\in Hx$.
\end{enumerate}
\end{lemma}
\begin{proof}
It is clear that \ref{lem-1-3} implies \ref{lem-1-1}, and \ref{lem-1-1} implies \ref{lem-1-2}.

Suppose that \ref{lem-1-2} holds. For all $hx\in Hx$ with $h\in H$, since $H$ is normal, we have $(hx)^{-1}=x^{-1}h^{-1}\in x^{-1}H=Hx^{-1}=Hx$. Then \ref{lem-1-4} is valid.

Suppose that \ref{lem-1-4} holds. Let $y\in Hx$. Then there exists $h\in H$ such that $y=hx$. Since $Hx$ is inverse-closed, there exists $h'\in H$ such that $y^{-1}=h'x$, which implies $y=x^{-1}h'^{-1}$. It follows that $y^2=hxx^{-1}h'^{-1}=hh'^{-1}\in H$. Since $y\in Hx$ was arbitrary, \ref{lem-1-3} is valid.

This completes the proof of this lemma.
\end{proof}

For an element $x\in G$, let $o(x)$ denote the {\em order} of $x$, that is, the smallest positive integer $m$ such that $x^m=e$. An element $x$ is called an {\em involution} if $o(x)=2$.

\begin{lemma}\label{lem-1.2}
Let $H$ be a normal subgroup of $G$, and $x$ an element of $G$. Then $x^2\notin H$ if and only if $y^2\notin H$ for each $y\in Hx$. Moreover, if $x^2\notin H$, then $Hx\cup Hx^{-1}$ is inverse-closed and has no involution.
\end{lemma}
\begin{proof}
For the first statement, it suffices to prove the necessity. Suppose, to the contrary that there exists $hx\in Hx$ such that $h\in H$ and $(hx)^2\in H$. Since $H$ is normal, there exists $h'\in H$ such that $hx=xh'$. Since $hx^2h'=(hx)^2\in H$, one gets $x^2\in H$, a contradiction.

Now suppose $x^2\notin H$. Let $y\in Hx\cup Hx^{-1}$. Since $H$ is normal, there exist $h'\in H$ and $\epsilon\in\{\pm1\}$ such that $y=x^{\epsilon}h'$, which implies $y^{-1}=h'^{-1}x^{-\epsilon}\in Hx\cup Hx^{-1}$. Since $y\in Hx\cup Hx^{-1}$ was arbitrary, $Hx\cup Hx^{-1}$ is inverse-closed. Since $x^2\notin H$, one has $x^{-2}\notin H$. By the first statement, $Hx\cup Hx^{-1}$ has no involution.
\end{proof}

Let $H$ be a nonempty subset of $G$. Write
\[
H^2=\{h^2: h\in H\}.
\]
Remark that $H^2$ is consisting of all square elements of $H$ if $H$ is a subgroup of $G$.

\begin{lemma}\label{o-og}
Suppose that $G$ is an abelian group of odd order and $H$ is a subgroup of $G$. If $g^2\in H$ for some $g\in G$, then $g\in H$.
\end{lemma}
\begin{proof}
Since $G$ has odd order, we have $H^2=H$. As a result, $g^2=h^2$ for some $h\in H$. It follows that $(gh^{-1})^2=e$, and so $g=h\in H$, as desired.
\end{proof}

The following result determines the structures of $\Gamma_{G,H}$ and $\Gamma_{G,H}^{+}$ with the case where the normal subgroup $H$ is {\em trivial} (either $\{e\}$ or $G$), which is analogous to \cite[Theorem 1]{Cam22}.

\begin{thm}\label{s-th1}
The following hold:
\begin{enumerate}
\item\label{s-th1-1} $\Gamma_{G,\{e\}}$ is the empty graph with $|G|$ vertices, while $\Gamma_{G,\{e\}}^{+}$ is a union of some independent edges and isolated vertices where $\{x,y\}\in E(\Gamma_{G,\{e\}}^{+})$ if and only if $x$ is the inverse of $y$;

\item\label{s-th1-2} $\Gamma_{G,G}^{+}$ is the complete graph with $|G|$ vertices, and $\Gamma_{G,G}$ is obtained by deleting all edges $\{x,x^{-1}\}$ in $\Gamma_{G,G}^{+}$ where $o(x)\ge 3$. In particular, the identity $e$ is adjacent to any other vertex in $\Gamma_{G,G}$.
\end{enumerate}
\end{thm}

These graphs considered in Theorem~\ref{s-th1} are not  interesting, so where necessary below we assume that $H$ is a  non-trivial normal subgroup of $G$.

By Lemmas~\ref{lem-1} and \ref{lem-1.2}, we give
the basic structure of $\Gamma_{G,H}$ and $\Gamma_{G,H}^{+}$ in the next result, which is similar to \cite[Theorem 2]{Cam22}.

\begin{thm}\label{s-th2}
Let $H$ be a non-trivial normal subgroup of $G$ with $|H|=t$. Then the following hold:
\begin{enumerate}
  \item\label{s-th2-1} Every connected component of $\Gamma_{G,H}^{+}$ is isomorphic to either the complete graph $K_{t}$ or the complete bipartite graph $K_{t,t}$. In particular, if $x^2\notin H$, then the subgraph induced by $Hx\cup Hx^{-1}$ is a connected component isomorphic to $K_{t,t}$; if $x^2\in H$, then the subgraph induced by $Hx$ is a connected component isomorphic to $K_{t}$;
  \item\label{s-th2-2} $\Gamma_{G,H}$ is obtained from $\Gamma_{G,H}^{+}$ by deleting a perfect matching from every connected component which is complete bipartite, and deleting a matching from a complete connected component on a coset $Hx$ covering all elements other than elements of $G^2$. In particular, in the  connected component of $\Gamma_{G,H}$ induced by $Hx$ with $x^2\in H$, any element of $G^2$ is adjacent to any other vertex.
\end{enumerate}
\end{thm}

\section{Perfect codes of $\Gamma_{G,H}$}

A {\em right transversal} (resp. {\em left transversal}) of a subgroup $H$ in $G$ is defined as a subset of $G$ which contains exactly one element in each right coset (resp. left coset) of $H$ in $G$. In an abelian group, every right coset of any subgroup is also a left coset of the subgroup, for the sake of simplicity, we then use the term ``transversal'' to substitute for ``right transversal'' or ``left transversal''.

The main result of this section is the following theorem which gives a necessary and sufficient condition for which normal subgroups $H$ of $G$ satisfy that $\Gamma_{G,H}$ admits a perfect code.

\begin{thm}\label{thm-1}
Let $H$ be a normal subgroup of $G$. Then $\Gamma_{G,H}$ admits a perfect code if and only if one of the following holds:
\begin{enumerate}
\item\label{thm-1-1} $H=\{e\}$;

\item\label{thm-1-2} $|H|=2$;

\item\label{thm-1-3} $Hx$ has an involution for each $x\in G\setminus H$ with $x^2\in H$.
\end{enumerate}
\end{thm}
\begin{proof}
We first prove the necessity. Assume that $\Gamma_{G,H}$ admits a perfect code. Suppose, for a contradiction, that $|H|>2$ and $Hx$ has no involutions for some $x\in G\setminus H$ with $x^2\in H$. In view of Lemma~\ref{lem-1}, $Hx$ is inverse-closed. Since $Hx$ does not contain an involution and the identity element $e$, $|Hx|$ is even. Now Theorem~\ref{s-th2} implies that the subgraph of $\Gamma_{G,H}$ induced by $Hx$ is  a connected component of $\Gamma_{G,H}$ and obtained by deleting a perfect matching from a complete graph of size $|H|$. Since $|H|\ge 4$, the connected component induced by $Hx$ does not admit a perfect code. It follows that $\Gamma_{G,H}$ has no perfect codes, a contradiction.

Next we prove the suffciency. If \ref{thm-1-1} holds, by Theorem \ref{s-th1} \ref{s-th1-1}, then $\Gamma_{G,H}$ admits a perfect code. Suppose that \ref{thm-1-2} holds. It follows that every vertex in $\Gamma_{G,H}$ has degree $1$ or $0$. Thus, $\Gamma_{G,H}$ is a union of some independent edges and isolated vertices. As a result, $\Gamma_{G,H}$ admits a perfect code.

Now assume that \ref{thm-1-3} holds. Let $\{x_1,x_2,\ldots,x_k\}$ be a right transversal of $H$ in $G$ with $k\geq1$. Without loss of generality, we may assume $x_1=e$, $x_i^2\in H$ for $2\le i \le t$, and $x_j^2\notin H$ for $t+1\le j \le k$, where $1\leq t\leq k$.

Since $x_i^2\in H$ for $1\leq i\leq t$, from Theorem~\ref{s-th2}, the subgraph induced by $Hx_i$ is a connected component of $\Gamma_{G,H}$. Since $x_i\in G\setminus H$ for $2\leq i\leq t$, $Hx_i$ has at least an involution, say $u_i$.

Since $x_j^2\notin H$ for $t+1\le j \le k$, from Lemma~\ref{lem-1}, we have $Hx_j\ne Hx_j^{-1}$, and so $k-t$ is even. By Theorem~\ref{s-th2},
the subgraph induced by $Hx_j\cup Hx_j^{-1}$ is a connected component of $\Gamma_{G,H}$. Note that $x_j$ is not an involution. Then $$N_{\Gamma_{G,H}}(x_j)=\{hx_j^{-1}:h\in H\setminus{\{e\}}\},
~~N_{\Gamma_{G,H}}(x_j^{-1})=\{hx_j:h\in H\setminus{\{e\}}\}.$$

Since $k-t$ is even and $Hx_j\ne Hx_j^{-1}$ for $t+1\le j \le k$, we may assume $k-t=2m$ and $Hx_{t+l}^{-1}=Hx_{t+m+l}$ for every $1\le l \le m$. Let
$$
C=\{e,u_2,\ldots,u_t,x_{t+1},x_{t+1}^{-1}, \ldots, x_{t+m},x_{t+m}^{-1}\}.
$$
Then by Theorem~\ref{s-th2}, it is easy to see that $C$ is a perfect code of $\Gamma_{G,H}$.
\end{proof}

\begin{cor}\label{cor-2}
Let $H$ be a normal subgroup of $G$.
If $|H|$ is an odd integer, then $\Gamma_{G,H}$ admits a perfect code.
\end{cor}
\begin{proof}
If $H=\{e\}$, from Theorem~\ref{thm-1} \ref{thm-1-1}, then the desired result holds. Now assume that $|H|\geq3$. By Theorem~\ref{thm-1},
we only need to consider the case that there exists $x\in G\setminus H$ with $x^2\in H$. It suffices to show that $Hx$ has an involution. Now it follows from Lemma~\ref{lem-1} that $Hx$ is inverse-closed. Since $e\notin Hx$ and $|Hx|$ is odd, $Hx$ has at least an involution, as desired.
\end{proof}

\begin{cor}\label{cor-3}
If $G$ has odd order, then $\Gamma_{G,H}$ admits a perfect code for any normal subgroup $H$ of $G$.
\end{cor}

\subsection{Abelian groups}

In this subsection, we always assume that $A$ is an abelian, additively written, group with identity $0$. Denote by $A_2$ the Sylow $2$-subgroup of $A$.

\begin{prop}\label{p-abg}
Let $H$ be a subgroup of $A$. If $\Gamma_{A,H}$ admits a perfect code, then $\Gamma_{A_2,H_2}$ admits a perfect code.
\end{prop}
\begin{proof}
If $|H_2|\le 2$, by Theorem~\ref{thm-1} \ref{thm-1-1} and \ref{thm-1-2}, then the desired result holds. We only need to consider the case that $|H_2|\geq3$ and there exists $x\in A_2\setminus H_2$ such that $2x\in H_2$. Let $$A=A_2\times Q,~~H=H_2\times Q_1,$$ where $Q$ is an abelian group of odd order and $Q_1\le Q$.
Then $(x,e')\in (A_2\times Q)\setminus (H_2\times Q_1)$ where $e'$ is the identity element of $Q$, and $2(x,e')\in (H_2\times Q_1)$. Since $\Gamma_{A,H}$ admits a perfect code, from Theorem~\ref{thm-1}, there is an involution in $(H_2\times Q_1)+(x,e')$,  say $(h+x,y'+e')$ for some $h\in H_2$ and $y'\in Q_1$. It means that $(2(h+x),2y')=(e,e')$ where $e$ is the identity element of $A_2$. The fact that $Q_1$ has odd order implies $y'=e'$. Since $(h+x,e')$ is an involution, $h+x$ is an involution belonging to $H_2+x$. By Theorem~\ref{thm-1} \ref{thm-1-3} again, $\Gamma_{A_2,H_2}$ admits a perfect code, as desired.
\end{proof}

\begin{prop}\label{p2-abg}
Let $H$ be a subgroup of $A$ with $|H_2|\ne 2$. If $\Gamma_{A_2,H_2}$ admits a perfect code, then $\Gamma_{A,H}$ admits a perfect code.
\end{prop}
\begin{proof}
If $|H_2|=1$, then $|H|$ is odd, and $\Gamma_{A,H}$ admits a perfect code from Corollary~\ref{cor-2}, as desired.  Let $$A=A_2\times Q,~~H=H_2\times Q_1,$$ where $Q$ is an abelian group of odd order and $Q_1\le Q$. By Theorem \ref{thm-1}, we only need to consider the case that $|H_2|\geq3$ and $(2x,2y)\in H_2\times Q_1$ for some $(x,y)\in (A_2\times Q)\setminus (H_2\times Q_1)$.  It follows from Lemma~\ref{o-og} that $y\in Q_1$. As a result, we have $x\notin H_2$. Thus, we have that $x\in A_2\setminus H_2$ and $2x\in H_2$. Now Theorem~\ref{thm-1} implies that $H_2+x$ has an involution, say $h+x$ for some $h\in H_2$.
Since $y\in Q_1$, we have $(h,-y)\in H_2\times Q_1$, and so $(h+x,e')=(h,-y)+(x,y)\in(H_2\times Q_1)+(x,y)$ is an involution, where $e'$ is the identity element of $Q$.  Theorem~\ref{thm-1} \ref{thm-1-3} implies that $\Gamma_{A,H}$ admits a perfect code, as desired.
\end{proof}

The next corollary follows immediately from Propositions \ref{p-abg} and \ref{p2-abg}.

\begin{cor}
Let $H$ be a subgroup of $A$ with $|H_2|\ne 2$. Then $\Gamma_{A_2,H_2}$ admits a perfect code if and only if $\Gamma_{A,H}$ admits a perfect code.
\end{cor}

 For the cyclic group $\mathbb{Z}_n$ of order $n$, we say  $\mathbb{Z}_n=\{0,1,2,\ldots,n-1\}$, where $0$ is its identity element.

\begin{remark}
If we delete the condition ``$|H_2|\ne 2$" in  Proposition~\ref{p2-abg}, then it is false. For example, let
$$
A=\mathbb{Z}_4\times \mathbb{Z}_3\times \mathbb{Z}_3,
~~H=\{(0,1,0),(0,2,0),(0,0,0),(2,1,0),(2,2,0),(2,0,0)\}.
$$
Then $H\le A$, $H\cong \mathbb{Z}_2\times \mathbb{Z}_3$ and $|H_2|=2$. By Theorem~\ref{thm-1}, $\Gamma_{A_2,H_2}$ admits a perfect code. However, taking $(1,0,0)\in A\setminus H$ with $2(1,0,0)\in H$, we have that $H+(1,0,0)$ has no involutions, and by Theorem~\ref{thm-1}, $\Gamma_{A,H}$ does not admit a perfect code.
\end{remark}

Let $H$ be an abelian $2$-group of order at least $2$. Then
by fundamental theorem of finite abelian groups, we may assume that
\begin{align}\label{abe-2-fg}
H=\mathbb{Z}_{2^{n_1}}\times \mathbb{Z}_{2^{n_2}} \times \cdots \times \mathbb{Z}_{2^{n_k}},
\end{align}
where $1\le n_i\le n_j$ for all $1\le i \le j\le k$.
Note that the abelian group  as presented in \eqref{abe-2-fg} is non-cyclic if and only if $k\ge 2$.

The following result determines all subgroups of a non-cyclic abelian $2$-group such that a subgroup sum graph of this non-cyclic abelian $2$-group group with respect to such a subgroup admits a perfect code.

\begin{thm}\label{ncg-2g}
Let $H$ be a non-cyclic abelian $2$-group as presented in \eqref{abe-2-fg} and $K$ a subgroup of $H$ with $|K|\geq3$. Then  $\Gamma_{H,K}$ admits a perfect code if and only if
$K$ has the following form:
\begin{align}\label{pc-2-ag}
K=K_1\times K_2\times\cdots \times K_k,
\end{align}
where $K_i$ is a trivial subgroup of $\mathbb{Z}_{2^{n_i}}$ for $1\leq i\leq k$.
\end{thm}
\begin{proof}
Note that $k\ge 2$. We first prove the sufficiency. Now suppose that $K$ is a group having the form \eqref{pc-2-ag}. In the following, we assume that $|K|\geq3$. Without loss of generality, we may assume $K_i=\{0\}$ for $1<i \le t-1$ and $K_j=\mathbb{Z}_{2^{n_j}}$ for $t\le j \le k$, where $1<t \le k$. Suppose that there exists $x\in H\setminus K$ such that $2x\in K$.  Let $x=(b_1,b_2,\cdots,b_k)$. Then $b_i=2^{n_i-1}$ or $0$ for any $1\le i \le t-1$, and there exists at least an index $i\in \{1,\ldots,t-1\}$ such that $b_i=2^{n_i-1}$. Since for any $t\le j \le k$, $(0,\ldots,0,-b_t,-b_{t+1},\ldots,-b_k)\in K$. It follows that $(b_1,\cdots,b_{t-1},0,\cdots,0)\in K+x$ and $(b_1,\cdots,b_{t-1},0,\cdots,0)$ is an involution. Thus, $\Gamma_{H,K}$ admits a perfect code from Theorem~\ref{thm-1} \ref{thm-1-3}.

Now let $K$ be a subgroup of $H$ such that $K$ is not isomorphic to a group having the form \eqref{pc-2-ag}.
Up to isomorphism, let
\begin{align}
K\cong\mathbb{Z}_{2^{m_1}}\times \mathbb{Z}_{2^{m_2}} \times \cdots \times \mathbb{Z}_{2^{m_k}},\nonumber
\end{align}
where for all $1\le i \le k$ and $m_i\le n_i$.
As a result, there exists $l\in \{1,2,\ldots,k\}$ such that $1\le m_l<n_l$. Note that $K\le H$.
It follows that if $(b_1,\ldots,b_l,\ldots,b_k)\in K$, then $b_l=k\cdot2^{n_l-m_l}$ with $0\le k \le 2^{m_l}-1$. Let
$$
x=(\underbrace{0,\cdots,0}_{l-1},2^{n_l-m_l-1},0,\ldots,0).
$$
Then $x\in H\setminus K$ and $2x\in K$. If $2^{n_l-m_l-1}+a\cdot2^{n_l-m_l}\equiv2^{n_l-1}~({\rm mod}~2^{n_l})$ for some $a\in\{0,1\ldots,2^{m_l}-1\}$, then $1+2a\equiv2^{m_l}~({\rm mod}~2^{m_l+1})$, which is impossible. Hence, $2^{n_l-m_l-1}+a\cdot2^{n_l-m_l}\not\equiv 2^{n_l-1}~({\rm mod}~2^{n_l})$ for any $0\le a \le 2^{m_l}-1$.
Also, it is clear that $2^{n_l}\nmid(2^{n_l-m_l-1}+a\cdot2^{n_l-m_l})$ for any $0\le a \le 2^{m_l}-1$.
As a result, $K+x$ has no involutions. By Theorem~\ref{thm-1}, $\Gamma_{H,K}$ has no perfect codes.
\end{proof}

The following result determines all subgroups $H$ of $\mathbb{Z}_n$ such that $\Gamma_{\mathbb{Z}_n,H}$ admits a perfect code.

\begin{thm}\label{th-cg}
Suppose that $H=\langle a\rangle$ is a subgroup of $\mathbb{Z}_n$, where $a$ is the smallest positive integer in $H$. Then $\Gamma_{\mathbb{Z}_n,H}$ admits a perfect code if and only if one of the following holds:
\begin{enumerate}
  \item\label{th-cg-1} $n$ is odd;
  \item\label{th-cg-2} $n$ is even and $o(a)$ is odd;
  \item\label{th-cg-3} $n$ is even and $o(a)=2$;
  \item\label{th-cg-4} $n$ is even, $o(a)\ge 4$ is even and $a$ is odd.
\end{enumerate}
\end{thm}
\begin{proof}
In view of Corollaries~\ref{cor-2} and \ref{cor-3}, and Theorem~\ref{thm-1} \ref{thm-1-2}, if \ref{th-cg-1}, \ref{th-cg-2} or \ref{th-cg-3} holds, then $\Gamma_{\mathbb{Z}_n,H}$ admits a perfect code. We only need to consider the case that $n$ is even and $o(a)\ge 4$ is even. It suffices to show that $\Gamma_{\mathbb{Z}_n,H}$ admits a perfect code if and only if $a$ is odd.

Suppose that $a$ is odd.
If $H=\mathbb{Z}_n$, by Theorem~\ref{thm-1} \ref{thm-1-3}, then $\Gamma_{\mathbb{Z}_n,H}$ admits a perfect code, as desired. We only need to consider the case that $\langle a\rangle$ is a non-trivial subgroup of $\mathbb{Z}_n$. It follows that $a>1$. Let $i$ be an integer such that $i\in\mathbb{Z}_n\setminus H$. Then $i\equiv j~({\rm mod}~a)$ for some $j\in\{1,2,\ldots,a-1\}$. Therefore, $2i\equiv 2j~({\rm mod}~a)$. Since $a$ is odd, one gets $a\nmid 2i$, and so $2i\notin H$. It follows from Theorem~\ref{thm-1} that $\Gamma_{\mathbb{Z}_n,H}$ admits a perfect code, as desired.

Suppose that $a$ is even. Note that $a\ge 2$. It follows that $a/2\notin H$. Since $\mathbb{Z}_n$ has the unique involution belonging to $H$, $H+(a/2)$ has no involutions. By Theorem~\ref{thm-1}, $\Gamma_{\mathbb{Z}_n,H}$ does not admit a perfect code, as desired.
\end{proof}

\begin{cor}\label{2g-cyc}
Let $K$ be a subgroup of $\mathbb{Z}_{2^k}$, where $k\ge 1$. Then $\Gamma_{\mathbb{Z}_{2^k},K}$ admits a perfect code if and only if $|K|=1,2$ or $2^k$.
\end{cor}

We use the following example to illustrate Theorem \ref{th-cg}.

\begin{example}
Consider the cyclic group $\mathbb{Z}_{60}$ of order $60$. By Theorem~\ref{th-cg}, if $H\in \{\langle 2\rangle,\langle 6\rangle,\langle 10\rangle\}$, then
$\Gamma_{\mathbb{Z}_n,H}$ does not admit a perfect code. Otherwise,  $\Gamma_{\mathbb{Z}_n,H}$ admits a perfect code.
\end{example}

\subsection{Dihedral groups}

Let $n\ge 3$ be a positive integer. The {\em dihedral group} $D_{2n}$ of order $2n$ has a presentation:
\begin{align}\label{d2n}
D_{2n}=\langle a,b: a^n=b^2=e,~bab=a^{-1}\rangle.
\end{align}
Remark that
\begin{align}\label{d2n-1}
D_{2n}=\langle a\rangle \cup \{ab,a^2b,\ldots,a^{n-1}b,b\},~~
o(a^ib)=2 \text{ for any }0\le i \le n-1.
\end{align}

The following result is a basic fact, we omit its proof.

\begin{lemma}\label{dihe-nsg}
Let $D_{2n}$ be the dihedral group as presented in \eqref{d2n}. Then a proper subgroup $H$ of $D_{2n}$ is normal if and only if one of the following holds:
\begin{enumerate}
  \item\label{dihe-nsg-1} $H\le \langle a\rangle$ if $n$ is odd;
  \item\label{dihe-nsg-2} $H\le \langle a\rangle$, $H=\langle a^2,b\rangle$ or $H=\langle a^2,ab\rangle$ if $n$ is even.
\end{enumerate}
\end{lemma}

The following result determines all subgroups $H$ of $D_{2n}$ such that $\Gamma_{D_{2n},H}$ admits a perfect code.

\begin{thm}\label{pc-dihe}
Let $D_{2n}$ be the dihedral group as presented in \eqref{d2n} and $H$ a normal subgroup of $D_{2n}$.
Then $\Gamma_{D_{2n},H}$ admits a perfect code if and only if either $H=D_{2n}$ or one of the following holds:
\begin{enumerate}
  \item\label{pc-dihe-1} $H\le \langle a\rangle$ if $n$ is odd;
  \item\label{pc-dihe-2} $H=\langle a^2,b\rangle$, $H=\langle a^2,ab\rangle$, or $H=\langle a^t\rangle$ where either $n/t$ is odd or $2$, or $n/t\ge 4$ is even and $t$ is odd, if $n$ is even.
\end{enumerate}
\end{thm}
\begin{proof}
By Theorem~\ref{thm-1}, we only need to consider that $H$ is non-trivial. If $n$ is odd, by Lemma~\ref{dihe-nsg}, then every non-trivial normal subgroup of $D_{2n}$ has odd order, which implies that the desired result holds from Corollary~\ref{cor-2}.  In the following, we assume that $n$ is even. By Lemma~\ref{dihe-nsg}, if $H=\langle a^2,a^{i}b\rangle$ for $i\in\{0,1\}$, then $|H|=n$ and $D_{2n}=H\cup H(a^{i+1}b)$, which imply that $\Gamma_{D_{2n},H}$ admits a perfect code since $o(a^{i+1}b)=2$ from Theorem~\ref{thm-1}. Now by Lemma~\ref{dihe-nsg}, let $H\le \langle a\rangle$. Note that \eqref{d2n-1} implies that every element in $D_{2n}\setminus \langle a\rangle$ is an involution. It follows from Theorem~\ref{thm-1} that $\Gamma_{D_{2n},H}$ admits a perfect code if and only if $\Gamma_{\langle a\rangle,H}$ admits a perfect code. Now the required result follows from Theorem~\ref{th-cg}.
\end{proof}

\subsection{Dicyclic groups}
For $n\ge 2$, the {\em dicyclic group} ${\rm Dic}_n$ is a group of order $4n$, which has a presentation as follows:
\begin{align}\label{q4m}
{\rm Dic}_n=\langle a,b: a^n=b^2, a^{2n}=e, b^{-1}ab=a^{-1}\rangle.
\end{align}
Remark that
\begin{align}\label{q4m-1}
{\rm Dic}_n=\langle a\rangle \cup \{a^ib: 1\le i \le 2n\},~~
o(a^ib)=4 \text{ for any } 1\le i \le 2n.
\end{align}

\begin{lemma}{\rm (\cite[p. 420]{JL01})}\label{conjq4n}
Let ${\rm Dic}_{n}$ be the dicyclic group as presented in \eqref{q4m}. Then ${\rm Dic}_n$ has $n+3$ conjugacy classes as follows:
$$
\{e\},\{a^n\},\{a^{i},a^{-i}\},
b^{{\rm Dic}_n}=\{a^{2j}b: 0\le j \le n-1\},(ab)^{{\rm Dic}_n}=\{a^{2j+1}b: 0\le j \le n-1\},
$$
where $1\le i \le n-1$.
\end{lemma}

The following result determines all subgroups $H$ of ${\rm Dic}_{n}$ such that $\Gamma_{{\rm Dic}_{n},H}$ admits a perfect code.

\begin{thm}\label{pc-dic}
Let ${\rm Dic}_n$ be the dicyclic group as presented in \eqref{q4m} and $H$ a normal subgroup of ${\rm Dic}_n$.
Then $\Gamma_{{\rm Dic}_n,H}$ admits a perfect code if and only if either $H={\rm Dic}_n$ or $H=\langle a^t\rangle$ such that $(2n)/t$ is an odd integer or $2$.
\end{thm}
\begin{proof}
By Theorem~\ref{thm-1}, we only need to consider that $H$ is non-trivial.
Note that by \eqref{q4m}, it is clear that $\langle a\rangle$ is a normal subgroup of ${\rm Dic}_n$, which implies that every subgroup of $\langle a\rangle$ is also a normal subgroup of ${\rm Dic}_n$. In view of \eqref{q4m-1}, $a^n$ is the unique involution in ${\rm Dic}_n$.

Suppose that $H$ contains an element $a^ib$ for some $i\in\{0,1,\ldots,2n-1\}$. By Lemma~\ref{conjq4n}, we have $(a^{i}b)^{{\rm Dic}_n}\subseteq H$. In view of \eqref{q4m-1}, one gets $(a^ib)^2=a^n\in H$. Since $|H|\mid 2n$ and $H$ is non-trivial, we obtain $a^{i+1}b\notin H$. Since $(a^{i+1}b)^2=a^n\in H$ and $H(a^{i+1}b)$ has no involutions, from Theorem~\ref{thm-1}, $\Gamma_{{\rm Dic}_n,H}$ has no perfect codes.

Suppose that $H\le \langle a\rangle$ such that $|H|\ge 4$ is even. Since $a^n$ is the unique involution in ${\rm Dic}_n$, one has $a^n\in H$. By \eqref{q4m-1}, $b^2=a^n$. Since $b\in {\rm Dic}_n\setminus H$ and $Hb$ has no involutions, form Theorem~\ref{thm-1}, $\Gamma_{{\rm Dic}_n,H}$ has no perfect codes.

This completes the proof of this result.
\end{proof}

\section{Code-perfect groups}

In this section, we concern code-perfect groups. The following result is obtained by applying Theorem~\ref{thm-1} and  Corollary~\ref{cor-3} to abelian groups.

\begin{prop}\label{cor-1}
Let $A$ be an abelian group. Then $\Gamma_{A,H}$ admits a perfect code for each subgroup $H$ of $A$ if one the following holds:
\begin{enumerate}
\item\label{cor-1-1}$A$ has odd order;
\item\label{cor-1-2} $A$ is an elementary abelian $2$-group.
\end{enumerate}
\end{prop}

The next result is immediate from Lemma \ref{dihe-nsg} \ref{dihe-nsg-1} and Theorem \ref{pc-dihe} \ref{pc-dihe-1}.

\begin{prop}\label{cor-2'}
Let $D_{2n}$ be the dihedral group of order $2n$. If $n\geq3$ is odd, then $\Gamma_{D_{2n},H}$ admits a perfect code for each subgroup $H$ of $D_{2n}$.
\end{prop}

By Propositions \ref{cor-1} and \ref{cor-2'}, there exist code-perfect groups. In this section, we classify all finite code-perfect Dedekind groups. Suppose that $G$ is a Dedekind group.
Then $G$ is a code-perfect group if and only if for any  subgroup $H$ of $G$, $\Gamma_{G,H}$ admits a perfect code. Note that all abelian groups and the quaternion group $Q_8$ of order $8$ are Dedekind groups. The following remark gives some examples of Dedekind groups which are not code-perfect.

\begin{remark}\label{Q8}
The following hold:
\begin{enumerate}
\item\label{Q8-1} Let $Q_8$ be the quaternion group of order $8$. Namely,
$$
Q_8=\{\pm 1,\pm i,\pm j,\pm k\},~ij=k,~jk=i,~ki=j,~i^2=j^2=k^2=-1.
$$
Let $H=\langle i\rangle=\{\pm 1,\pm i\}$, which is a normal subgroup of $Q_8$. Then for any $x\in Q_8\setminus H$, we have $x^2=-1\in H$, however, $Hx=\{\pm j,\pm k\}$ has no involutions. It follows from Theorem~\ref{thm-1} that $\Gamma_{Q_8,H}$ does not admit a  perfect code.
In fact, the result also can be obtained from Theorem~\ref{pc-dic} as $Q_8\cong {\rm Dic}_2$.

\item\label{Q8-2} Let $A=\mathbb{Z}_2\times \mathbb{Z}_4$.
Take the subgroup $H=\{(0,0),(0,2),(1,0),(1,2)\}$ in $A$. In fact, $H\cong \mathbb{Z}_2\times \mathbb{Z}_2$. Then it follows from Theorem~\ref{ncg-2g} that $\Gamma_{A,H}$ has no perfect codes.
\end{enumerate}
\end{remark}


Before giving the main result of this section, we need an auxiliary result which can be checked directly.

\begin{lemma}\label{sub-cp}
Suppose that $G$ is a Dedekind group. Then $G$ is code-perfect if and only if for any subgroup $H$ of $G$, $H$ is code-perfect.
\end{lemma}

The main result of this section is as follows, which classifies all code-perfect Dedekind group.

\begin{thm}\label{code-perfect}
A Dedekind group is code-perfect if and only if it is isomorphic to either $\mathbb{Z}_4$ or $\mathbb{Z}_2^t\times Q$, where $t\ge 0$ and $Q$ is an abelian group of odd order.
\end{thm}
\begin{proof}
We first prove the sufficiency. By Corollary~\ref{2g-cyc} and Proposition~\ref{cor-1} \ref{cor-1-1}, $\mathbb{Z}_4$ and all abelian groups with odd order are code-perfect, as desired. Let $H$ be a subgroup of $\mathbb{Z}_2^t\times Q$, where $t>0$ and $Q$ is an abelian group of odd order. It suffices to show that $\Gamma_{\mathbb{Z}_2^t\times Q,H}$ admits a perfect code.  By Theorem~\ref{thm-1} and Corollary~\ref{cor-2}, we only need to consider the case that $4\le |H|< |\mathbb{Z}_2^t\times Q|$, $|H|$ is even and $2a\in H$ for some $a\in (\mathbb{Z}_2^t\times Q)\setminus H$.
Write
$$H=H_2\times H_{2'},$$
where $H_2$ is the Sylow $2$-subgroup of $H$ and $H_{2'}$ is the  Hall $2'$-subgroup of $H$. Then $H_2\le \mathbb{Z}_2^t$ and $H_{2'}\le Q$. Let $a=x+x'$ for some $x\in \mathbb{Z}_2^t$ and $x'\in Q$. Then $2a=2x'\in H_{2'}\le Q$.

\medskip
\noindent {\bf Case 1.} $|H_{2'}|=1$.
\medskip

Since $2a=2x'\in H_{2'}\le Q$, we have $x'=e$, which implies $a\in \mathbb{Z}_2^t$. Therefore, $H+a\subseteq \mathbb{Z}_2^t$, and so $H+a$ must have an involution. Now it follows from Theorem~\ref{thm-1} \ref{thm-1-3} that $\Gamma_{G,H}$ admits a perfect code.

\medskip
\noindent {\bf Case 2.} $|H_{2'}|\ne 1$.
\medskip

Since $2a=2x'\in H_{2'}\le Q$, from Lemma~\ref{o-og}, one has $x'\in H_{2'}$. Let $y\in H_2$. Then $y-x'\in H$ and $y+x=y-x'+x+x'=y-x'+a\in H+a$. Since $x,y\in H_2\le \mathbb{Z}_2^t$, $y+x$ is an involution. It follows from Theorem~\ref{thm-1} \ref{thm-1-3} that $\Gamma_{\mathbb{Z}_2^t\times Q,H}$ admits a perfect code, as desired.

We next prove the necessity. Suppose that the group $K$ is code-perfect.
By \cite{Ded}, every non-abelian Dedekind group contains $Q_8$ as its a subgroup. Thus, it follows from Remark~\ref{Q8} \ref{Q8-1} and Lemma~\ref{sub-cp} that every non-abelian Dedekind group is not code-perfect. As a result, $K$ is an abelian group.
In view of Remark~\ref{Q8} \ref{Q8-2} and Lemma~\ref{sub-cp},
$K$ has no subgroups isomorphic to $\mathbb{Z}_2\times \mathbb{Z}_4$. Moreover, Corollary~\ref{2g-cyc} and Lemma~\ref{sub-cp} imply that $K$ has no subgroups isomorphic to $\mathbb{Z}_8$. It follows that $K$ is isomorphic to one of the following:
$$
\mathbb{Z}_4\times Q, ~~\mathbb{Z}_2^t\times Q,
$$
where $t\ge 0$ and $Q$ is an abelian group of odd order.

It suffices to show that if $K=\mathbb{Z}_4\times Q$, then $|Q|=1$.
Suppose, for a contradiction, that $K=\mathbb{Z}_4\times Q$ with $|Q|\ne 1$.
Then $K$ has a subgroup $H=\{(0,x),(2,x):x\in Q\}$, which is isomorphic to $\mathbb{Z}_2\times Q$. Note that $(1,e)\notin H$, where $e$ is the identity element of $Q$. Then $2(1,e)\in H$. However, $H+(1,e)$ has no involutions since $K$ has a unique involution belonging to $H$. It follows from Theorem~\ref{thm-1} that $\Gamma_{K,H}$ does not admit a perfect code. Thus, $K$ is not code-perfect, a contradiction.
\end{proof}

The following corollary is immediate from Theorem \ref{code-perfect}.

\begin{cor}
The cyclic group $\mathbb{Z}_n$ is code-perfect if and only if $n$ is odd or $n=2d$, where $d$ is $2$ or an odd integer.
\end{cor}

\section{Total perfect codes of $\Gamma_{G,H}$}

In this section, we discuss subgroup sum graphs having a total perfect code.

\begin{lemma}\label{tcp-g3-ex}
Let $H$ be a normal subgroup of $G$ with $|H|=3$. Then $x^2\in H$ and $Hx$ contains an element of order at least $3$ for any $x\in G\setminus H$
if and only if $G=\mathbb{Z}_2^n\times \mathbb{Z}_3$ and $H=\{e,(0,\ldots,0,1),(0,\ldots,0,2)\}$, where $n\ge 0$.
\end{lemma}
\begin{proof}
It is easy to verify the sufficiency, we only prove the necessity. Suppose that $x^2\in H$ and $Hx$ contains an element of order at least $3$ for any $x\in G\setminus H$.
It is clear that if $G=\mathbb{Z}_3$, then the required result follows. In the following, let $H\neq G$. Let $P$ and $Q$ be a Sylow $2$-subgroup and Hall $2'$-subgroup of $G$, respectively. If $x\in G\setminus H$ has odd order, then $x^2\in H$, it follows that $Q=H\cong \mathbb{Z}_3$, which also implies that $G$ has a unique subgroup $H$ of order $3$.
If $P$ has an element $y$ of order $4$, then $y^2\notin H$, contrary to the fact that $x^2\in H$ for all $x\in G\setminus H$. It follows that $P$ is an elementary abelian $2$-group. Thus, $G\cong \mathbb{Z}_2^n\ltimes\mathbb{Z}_3$ for some $n\ge 0$.
Now let $H=\langle a\rangle$. Take an involution $u\in P$.
Suppose for a contradiction that $a$ and $u$ do not commute. Since $uau\in H$, we have $uau=a^2=a^{-1}$. It follows that $\langle a,u\rangle\cong D_{6}$. Thus, the coset $Hu=\{u,au,a^2u\}$  consists of three involutions, a contradiction. We conclude that $a$ and $u$ commute, so $P$ is normal in $G$. This means that $G\cong \mathbb{Z}_2^n\times\mathbb{Z}_3$, as desired.
\end{proof}

The following theorem gives a characterization for normal subgroups $H$  of $G$ such that $\Gamma_{G,H}$ admits a total perfect code.

\begin{thm}\label{tpc-allg}
Let $H$ be a normal subgroup of $G$. Then $\Gamma_{G,H}$ admits a total perfect code if and only if one of the following holds:
\begin{enumerate}
  \item\label{tpc-allg-1} $|H|=2$ and $x$ is an involution for each $x\in G\setminus H$ with $x^2\in H$;
  \item\label{tpc-allg-2} $G=\mathbb{Z}_2^n\times \mathbb{Z}_3$ and $H=\{e,(0,\ldots,0,1),(0,\ldots,0,2)\}$, where $n\ge 0$.
\end{enumerate}
\end{thm}
\begin{proof}
We first prove the sufficiency. Suppose that \ref{tpc-allg-1} holds.
Let $H=\{e,a\}$. If $H=G$, then clearly, $\Gamma_{G,H}$ admits a total perfect code, as desired.
Now take $y\in G\setminus H$. If $y^2\in H$, then $y$ is an involution, $Hy=\{y,ay\}$ and $ayy=a$, which imply that the connected component of $\Gamma_{G,H}$ induced by $Hy$ is an independent edge.
If $y^2\notin H$, then by Theorem~\ref{s-th2}, the component of $\Gamma_{G,H}$ induced by $(Hy)\cup(Hy^{-1})$ is two independent edges, that is $\{ay,y^{-1}\}$ and $\{y,ay^{-1}\}$.
It follows that, in this case, $\Gamma_{G,H}$ is a union of some independent edges, and so $\Gamma_{G,H}$ has a total perfect code, as desired.

Next, suppose that \ref{tpc-allg-2} holds.
If $H=G$, then clearly, $\Gamma_{G,H}\cong K_{1,2}$ admits a total perfect code, as desired. In the following, let $x\in G\setminus H$. Then $x^2\in H$ and $Hx$ contains an element of order at least $3$. By Lemma~\ref{lem-1}, $Hx$ is inverse-closed, which implies that $Hx$ has precisely one involution. We conclude that the connected component of $\Gamma_{G,H}$ induced by $Hx$ is isomorphic to $K_{1,2}$. Note that the connected component of $\Gamma_{G,H}$ induced by $H$ is also isomorphic to $K_{1,2}$. Since $x\in G\setminus H$ was arbitrary, $\Gamma_{G,H}$ has a total perfect code, as desired.

We next prove the necessity.
Suppose that $\Gamma_{G,H}$ admits a total perfect code. Clearly, $|H|\ge 2$.
We first claim that $|H|\le 3$. Suppose, for a contradiction that $|H|\ge 4$. Consider the connected component $\Gamma$ of $\Gamma_{G,H}$ induced by $H$. Let $C$ be a total perfect code of $\Gamma$. Then it must be $\{e,a\}\subseteq C$ for some $a\in H\setminus\{e\}$. Since there exists an element $x\in H\setminus C$ such that $x\ne a^{-1}$, we have $\{e,a\}\subseteq N(x)$, a contradiction. Therefore, our claim is valid, that is, $|H|=2,3$. We consider the following two cases.

\medskip
\noindent {\bf Case 1.} $|H|=2$.
\medskip

Let $H=\{e,a\}$.
Suppose that there exists $x\in G\setminus H$ such that $x^{2}\in H$. Then the connected component of $\Gamma_{G,H}$ induced by $Hx$ must be an edge. It follows that $axx=a$, so $x^2=e$. Namely, $x$ is an involution, as desired.

\medskip
\noindent {\bf Case 2.} $|H|=3$.
\medskip

Let $H=\{e,a,a^2\}$.  If $H=G$, then clearly, $\Gamma_{G,H}\cong K_{1,2}$ admits a total perfect code. Assume there exists an element $x$ in $G\setminus H$. Suppose for a contradiction that $x^2\notin H$. Then by Theorem~\ref{s-th2},
the connected component of $\Gamma_{G,H}$ induced by $(Hx)\cup(Hx^{-1})$ is displayed in Figure~\ref{f11}.
\begin{figure}[hptb]
  \centering
  \includegraphics[width=2.8cm]{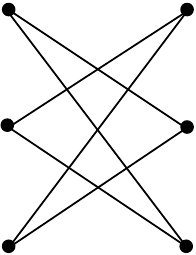}
  \caption{Induced subgroup by $(Hx)\cup(Hx^{-1})$}\label{f11}
\end{figure}
It is easy to see that the connected component of $\Gamma_{G,H}$ induced by $(Hx)\cup(Hx^{-1})$ has no total perfect codes, a contradiction. We conclude that $x^2\in H$.
If $Hx$ has no elements of order at least $3$, then every element of $Hx$ is an involution, which implies that the connected component of $\Gamma_{G,H}$ induced by $Hx$ is complete from Theorem \ref{s-th2}, this contradicts our hypothesis that $\Gamma_{G,H}$ admits a total perfect code. It follows that $x^2\in H$ and $Hx$ contains an element of order at least $3$ for any $x\in G\setminus H$. Now the desired result follows from Lemma~\ref{tcp-g3-ex}.
\end{proof}

\begin{remark}
Let $n=2q$ and $D_{2n}$ be the dihedral group as presented in \eqref{d2n}, where $q$ is an odd integer at least $3$. By Lemma \ref{dihe-nsg} \ref{dihe-nsg-2}, for any normal subgroup $H$ of order $2$,
\ref{tpc-allg-1} of Theorem~\ref{tpc-allg} occurs, so  $\Gamma_{D_{2n},H}$ admits a total perfect code.
\end{remark}

By applying Theorem~\ref{tpc-allg} to abelian groups, we classify all abelian groups $A$ and subgroups $H$ of $A$ such that $\Gamma_{A,H}$ admits a total perfect code.

\begin{thm}\label{tpc-1}
Let $A$ be an abelian group and $H$ a subgroup of $A$. Then $\Gamma_{A,H}$ admits a total perfect code if and only if one of the following holds:
\begin{enumerate}
  \item\label{tpc-1-1} $A=\mathbb{Z}_2^n\times Q,~|H|=2$, where $Q$ is an abelian group of odd order and $n\ge 1$;
  \item\label{tpc-1-2}
\begin{align}
A=\mathbb{Z}_2^n\times \mathbb{Z}_{2^{k_1}}\times \mathbb{Z}_{2^{k_2}}\times \dots \times\mathbb{Z}_{2^{k_t}}\times Q\nonumber
\end{align}
and
\begin{align}
H=\{e,(x_1,\ldots,x_n,z_1,\ldots,z_t,0)\},\nonumber
\end{align}
      where $x_i=0$ or $1$ for all $1\le i \le n$, $z_j=0$ or $2^{k_j-2}$ for all $1\leq j\leq t$, and there exists at least one $l\in \{1,\ldots,n\}$ such that $x_l=1$,
      $Q$ is an abelian group of odd order, $n,t\ge 1$, and $k_i\ge 2$ for all $1\le i \le t$;
  \item\label{tpc-1-3} $A=\mathbb{Z}_2^n\times \mathbb{Z}_3$ and $H=\{e,(0,\ldots,0,1),(0,\ldots,0,2)\}$, where $n\ge 0$.
\end{enumerate}
\end{thm}
\begin{proof}
By Theorem~\ref{tpc-allg}, it is easy to verify that the sufficiency is valid. We next prove the necessity. Assume that $\Gamma_{A,H}$ admits a total perfect code. Then, by Theorem~\ref{tpc-allg}, $|H|=2$ or $3$.
If $|H|=3$, then Theorem~\ref{tpc-allg} implies that \ref{tpc-1-3} holds.

We only need to consider the case $|H|=2$.
Let $A=A_2\times Q$, where $Q$ is an abelian group of odd order. Since $|A_2|\ge 2$, we have $A_2=\mathbb{Z}_2^n\times \mathbb{Z}_{2^{k_1}}\times \mathbb{Z}_{2^{k_2}}\times \dots \times\mathbb{Z}_{2^{k_t}}$ with $n+t\geq1$ and $k_i\ge 2$ for all $1\le i \le t$. If $t=0$, then \ref{tpc-1-1} holds. Assume that $t>0$. Then
$$H=\{e,(x_1,\ldots,x_n,z_1,\ldots,z_t,0)\},$$
where
$x_i=0$ or $1$ for all $1\le i \le n$, and $z_j=0$ or $2^{k_j-1}$ for all $1\le j \le t$, and there exists at least a positive integer $l$ such that $x_l=1$ with $1\leq l\leq n$, or $z_l=2^{k_l-1}$ with $1\le l \le t$.

Suppose that $n=0$ or $x_i=0$ for all $1\le i \le n$. Without loss of generality, we may assume $z_1=2^{k_1-2}$. Then let $y=(0,0,\ldots,0,2^{k_1-2},y_2,\ldots,y_t,0)$, where for all $2\le j \le t$,
\begin{align}\label{eq-tcp-ab-2}
y_j=\left\{
  \begin{array}{ll}
    0, & \hbox{if $z_j=0$;} \\
    2^{k_j-2}, & \hbox{if $z_j=2^{k_j-1}$.}
  \end{array}
\right.
\end{align}
It follows that $2y\in H$. However, $y$ is not an involution, contrary to Theorem~\ref{tpc-allg}. Thus, \ref{tpc-1-2} holds.
\end{proof}

\section{(Total) Perfect codes of $\Gamma_{G,H}^{+}$}

In this section, we consider extended subgroup sum graphs. The first result gives a characterization for normal subgroups $H$  of $G$ such that $\Gamma_{G,H}^+$ admits a perfect code.

\begin{thm}\label{pc-addg}
Let $H$ be a normal subgroup of $G$. Then
$\Gamma_{G,H}^{+}$ has a perfect code if and only if either $H$ is trivial or $G^2\subseteq H$.
\end{thm}
\begin{proof}
If $H$ is trivial, then $\Gamma_{G,H}^{+}$ has a perfect code by Theorem~\ref{s-th1}. Now suppose that $G^2\subseteq H$ and $H$ is non-trivial. Then for any $x\in G\setminus H$, we have $x^2\in H$. It follows from Theorem \ref{s-th2} \ref{s-th1-1} that any connected component of $\Gamma_{G,H}^{+}$ is the induced subgraph by a coset of $H$, which is isomorphic to $K_{|H|}$. As a consequence, $\Gamma_{G,H}^{+}$ has a perfect code, as desired.

Conversely,
suppose that $\Gamma_{G,H}^{+}$ has a perfect code. Assume that $H$ is non-trivial. It suffices to show that $G^2\subseteq H$.
Suppose for a contradiction that there exists $x\in G\setminus H$ such that $x^2\notin H$. Then by Theorem \ref{s-th2} \ref{s-th1-1}, the connected component of $\Gamma_{G,H}^{+}$ induced by $(Hx)\cup (Hx^{-1})$ is isomorphic to $K_{|H|,|H|}$, which implies that $\Gamma_{G,H}^{+}$ has no perfect codes as $|H|\ge 2$. We conclude $x^2\in H$. Namely, $G^2\subseteq H$, as desired.
\end{proof}

Let $A$ be an abelian group. Then $2A$ is a subgroup of $A$. In particular, if
$$A=\mathbb{Z}_{2^{m_1}}\times \mathbb{Z}_{2^{m_2}}\times \cdots \times \mathbb{Z}_{2^{m_t}}\times Q,$$
where $m_i\ge 1$ for each $1\le i\le t$ and $Q$ has odd order, then
\begin{align*}
  2A= & H_1\times H_2\times \cdots \times H_t\times Q \\
   =& \{(x_1,\ldots,x_t,y): x_i \text{ is even for all $1\le i \le t$},y\in Q\}
\end{align*}
where for any $1\le i\le t$, $H_i$ is the unique cyclic subgroup of $\mathbb{Z}_{2^{m_i}}$ which has order $2^{m_i-1}$.
In particular, if $A$ has odd order, then $2A=A$.

Applying Theorem~\ref{pc-addg} to abelian groups, we have the following corollary.

\begin{cor}
The following hold:
\begin{enumerate}
  \item\label{coro-1} If $G$ is an elementary abelian $2$-group, then $\Gamma_{G,H}^{+}$ has a perfect code for any subgroup $H$ of $G$;
  \item\label{coro-2} If $G$ is an abelian group of odd order, then $\Gamma_{G,H}^{+}$ has a perfect code if and only if $H$ is trivial;
  \item\label{coro-3} If $n$ is an even positive integer, then $\Gamma_{\mathbb{Z}_n,H}^{+}$ has a perfect code if and only if $H\in \{\mathbb{Z}_n,\{0\},\langle 2\rangle\}$;
  \item\label{coro-4} If $G=\mathbb{Z}_4^n$ for some $n\ge 1$, then $\Gamma_{G,H}^{+}$ has a perfect code if and only if $H$ is trivial or $H$ contains all involutions of $G$.
\end{enumerate}
\end{cor}

We conclude the paper by giving a necessary and sufficient condition for which normal subgroups $H$ of $G$ satisfy that $\Gamma_{G,H}^{+}$ has a total perfect code.

\begin{thm}
Let $H$ be a normal subgroup of $G$. Then $\Gamma_{G,H}^{+}$ has a total perfect code if and only if $|G|$ is even and $|H|=2$.
\end{thm}
\begin{proof}
The result trivially follows from Theorems~\ref{s-th1} and \ref{s-th2}.
\end{proof}

\section*{Acknowledgements}

X. Ma started thinking on this project during him visit to  Southern University of Science and Technology, and
he is very thankful to  Prof. Cai Heng Li for hospitality and moral support.
X. Ma is supported by the National Natural Science Foundation of China (Grant No. 12326333) and the Shaanxi Fundamental Science Research Project for Mathematics and Physics (Grant No. 22JSQ024). Y.~Yang is supported by the National Natural Science Foundation of China (Grant No. 12101575).

\section*{Data Availability Statement}

No data was used for the research described in the article.

\end{document}